\documentclass[12pt,twoside]{amsart}
\usepackage{amsmath}
\usepackage{amsthm}
\usepackage{amsfonts}
\usepackage{amssymb}
\usepackage{latexsym}
\usepackage{mathrsfs}
\usepackage{amsmath}
\usepackage{amsthm}
\usepackage{amsfonts}
\usepackage{amssymb}
\usepackage{latexsym}
\usepackage{geometry}
\usepackage{dsfont}
\usepackage[dvips]{graphicx}
\usepackage{color}
\usepackage[all]{xy}

\usepackage{amsthm,amsmath,amssymb}\usepackage{mathrsfs}

\date{}
\allowdisplaybreaks[4] \footskip=15pt
\renewcommand{\uppercasenonmath}[1]{}

\usepackage{graphicx,amssymb}
\usepackage[all]{xy}
\usepackage{amsmath}

\allowdisplaybreaks
\usepackage{amsthm}
\usepackage{color}

\theoremstyle{plain}
\newtheorem{theorem}{Theorem}[section]
\newtheorem{proposition}[theorem]{Proposition}
\newtheorem{lemma}[theorem]{Lemma}
\newtheorem{corollary}[theorem]{Corollary}
\theoremstyle{definition}
\newtheorem{example}[theorem]{Example}
\newtheorem{definition}[theorem]{Definition}
\newtheorem{question}[theorem]{Open Question}
\theoremstyle{definition}

\theoremstyle{remark}
\newtheorem{remark}[theorem]{Remark}

\newcommand{\pf}{\noindent\begin {proof}}
\newcommand{\epf}{\end{proof}}

\newcommand{\Ext}{\mbox{\rm Ext}}
\newcommand{\Hom}{\mbox{\rm Hom}}

\newcommand{\Prufer}{Pr\"{u}fer}

\newcommand{\FPR}{\mathcal{FPR}}

\newcommand{\Id}{\mathrm{Id}}

\newcommand{\Proj}{\mathcal{P}}

\def\fd{{\rm fd}}

\def\GV{{\rm GV}}
\def\tor{{\rm tor_{\rm GV}}}
\def\Hom{{\rm Hom}}
\def\Ext{{\rm Ext}}

\def\K{{\rm K}}

\def\H{{\rm H}}

\def\GV{{\rm GV}}
\def\Max{{\rm Max}}
\def\DW{{\rm DW}}

\def\gld{{\rm gld}}
\def\fPD{{\rm fPD}}
\def\fpD{{\rm fpD}}
\def\FPD{{\rm FPD}}

\def\pd{{\rm pd}}

\def\m{{\frak m}}

\def\grade{{\rm grade}}

\def\DQ{{\rm DQ}}

\def\id{{\rm id}}
\def\FP{{\rm FP\mbox{-}}}

\begin{document}
\begin{center}
{\large  \bf  The small finitistic dimensions of commutative rings, III}

\vspace{0.5cm}    Xiaolei Zhang$^{a}$\\

{\footnotesize a.\  School of Mathematics and Statistics, Tianshui Normal University, Tianshui 741001, China\\
 E-mail: zxlrghj@163.com\\}
\end{center}


\bigskip
\centerline { \bf  Abstract}   The small finitistic dimension $\fPD(R)$ of a ring $R$ is defined  to be the supremum of projective dimensions of $R$-modules with finite projective resolutions.  In this paper,  we show that a commutative ring $R$ has $\fPD(R)\leq d$ if and only if for any finitely generated ideal $I$ of $R$, if $\Ext_R^i(R/I,R)=0$ for each $i=0,\dots,d$, then $\Ext_R^i(R/I,R)=0$ for all $i\geq 0.$ As an application, we obtain that, for any commutative ring $R$, $\fPD(R)\leq \FP\Id_RR$, the latter is  the self-FP-injective dimension of $R$. We also give some applications of these results to (weak) $(n,d)$-rings, $\DW$-rings and rings of \Prufer\ type.
\bigskip
\leftskip10truemm \rightskip10truemm \noindent
\vbox to 0.3cm{}

\leftskip0truemm \rightskip0truemm
\bigskip
\section{introduction}
Throughout this paper, all rings are commutative rings with identity. Let $R$ be a ring. Let $S$ be a subset of $R$, we denote by $\langle S\rangle$ to be the ideal of $R$ generated by $S$.
Let $M$ be an $R$-module, we use $\pd_RM$ (resp., $\fd_RM$, $\id_R M$, $\FP\id_RM$) to denote the projective (resp., flat, injective, FP-injective) dimension of $M$ over $R$. Write $\gld(R)$ (resp., $w.\gld(R)$, $\FP\id_RR$) for the global dimension (resp., weak global dimenison, self-FP-injective dimension) of $R$. 

In ring theory, the homological dimension serves as a core invariant, measuring the complexity and structure of rings, and is essential for studying the homological algebra.
However, the classical homological dimensions of some common rings can be infinite. For example, a  Noetherian local ring has a finite (weak) global dimension only if it is regular.  To address this challenge and obtain meaningful finite invariants for a broader class of rings, Bass \cite{B60} introduced the concepts of the little finitistic (projective) dimension and the big finitistic projective dimension in his seminal work.   The little (resp., big) finitistic (projective) dimension of $R$, denoted by $\fpD(R)$ (resp., $\FPD(R)$), is defined to be the supremum of the projective dimensions of  all finitely generated (resp., all) $R$-modules $M$ with finite projective dimensions.

As to the little finitistic  dimension of a ring $R$, there are minimal advancements of it on the non-Noetherian rings since the syzygies of finitely generated modules are  not finitely generated over non-Noetherian rings in general.
To amend this gap, Glaz \cite{G89} proposed the notion of small finitistic dimension of a ring $R$, writted by $\fPD(R)$, which  is defined to be the supremum of projective dimensions of $R$-modules with finite projective resolutions (see Definition \ref{def-fPD}).

The studies of  small finitistic dimensions of commutative rings were motivated by two conjectures proposed by Glaz et al. \cite{CFFG14} who asked that is the small finitistic dimension of a \Prufer\ ring (resp., total ring of quotients) at most $1$ (resp., $0$)?   In 2020, Wang et al. \cite{wzcc20,fkxs20} characterized rings $R$ with $\fPD(R)=0$, and then gave an example of total ring of quotients with small finitistic dimensions larger than $1$ giving a negative answer to Glaz's questions. Furthermore, the authors \cite{z26,zw23} recently gave examples of total rings of quotients $R$ with $\fPD(R)=n$ for each $n\in \mathbb{N}\cup\{\infty\}$, completely denying Glaz's questions.

In fact, the authors \cite{z26,zw23} characterized the small finitistic dimension of rings by using special semi-regular  ideals, tilting theories and Koszul cohomologies in the two preceding articles of this series. However, these are also some questions can not be directly deduced by these characterizations. For example, it can easily be deduced by \cite[Corollary 5.5]{B62} that $\fPD(R)\leq \id_RR$ when $R$ is a Noetherian ring. Furthermore, we have the following question: 
\begin{center} \textbf{Question:} What's the relationship of  $\fPD(R)$  and  $(\FP)\id_RR$ for a general ring $R$?  
\end{center}
We know that a ring $R$ has $\fPD(R)\leq d$ if and only if any  finitely generated ideal $I$ that satisfies $\Ext_R^i(R/I,R)=0$ for each $i=0,\dots,d$ is exactly $R$ (see \cite[Theorem 3.1]{zw23}). However, from this equivalence, we don't know how $\Ext_R^i(R/I,R)$ behaves when $i$ is large enough when $\fPD(R)$ is finite, and this question is essential to study its self-(FP-)injective dimension. 

In this paper, to address the above question, we obtain a new characterization of rings with $\fPD(R)\leq d$, i.e., 
a ring $R$ has $\fPD(R)\leq d$ if and only if for any finitely generated ideal $I$ of $R$, if $\Ext_R^i(R/I,R)=0$ for each $i=0,\dots,d$, then $\Ext_R^i(R/I,R)=0$ for all $i\geq 0$ (see Theorem \ref{wnd}). We also give a characterization of rings with an infinite small finitistic dimension (see Corollary \ref{wnd-inf}). As applications, we show that $\fPD(R)\leq \FP\id_RR$ for any ring $R$ (see Theorem \ref{idRfpd}). We also give some applications of these results to (weak) $(n,d)$-rings, $\DW$-rings and rings of \Prufer\ type in Section 3.

\section{The main result}

 Let $R$ be a ring and $M$ be an $R$-module. Then $M$ is said to have a \emph{finite projective resolution}, denoted by $M\in\FPR$, if  there exist a non-negative integer $n$ and an exact sequence
$$0\rightarrow P_n\rightarrow P_{n-1}\rightarrow \dots\rightarrow P_1\rightarrow P_0\rightarrow M\rightarrow 0$$
with each $P_i$  finitely generated projective.

 In 1989, Glaz \cite{G89} introduced the notion of the small finitistic dimension of a ring $R$, making it more suitable for homologies in non-Noetherian rings than the little  finitistic dimension brought by Bass \cite{B60}.

\begin{definition}\label{def-fPD} \cite[Page 67]{G89}
The \emph{small finitistic (projective) dimension} of a ring $R$, denoted by $\fPD(R)$, is defined to be the supremum of the projective dimensions of $R$-modules in $\FPR$.
\end{definition}

If we denote $\Proj^{\leq n}$  to be the class of $R$-modules with projective dimensions at most $n$ in  $\FPR$, then $\fPD(R)\leq n$ if and only if $\FPR=\Proj^{\leq n}$.

Recently, Zhang \cite{z26}  characterized the small finitistic dimensions of a ring via the Koszul homology in the previous article of this series.  Since this paper also employs techniques related to the Koszul homology, we briefly review the relevant concepts.

Let $R$ be a ring, and $\textbf{x} = x_1, \dots, x_n$ be a finite sequence of elements in $R$. Let $\alpha = (i_1, \dots, i_p)$, $1 \leq i_1 < \cdots < i_p \leq n$, be an ascending sequence of integers with $0 \leq p \leq n$. The module $K_p(\textbf{x})$ is defined to be a free $R$-module with basis $e_{\alpha} = e_{i_1} \wedge \dots \wedge e_{i_p}$. Define an $R$-homomorphism
\begin{align*}
	d_p:&\ K_p(\textbf{x}) \rightarrow K_{p-1}(\textbf{x}) \\
	&\  e_{\alpha}\mapsto  \sum\limits_{j=1}^p (-1)^{j+1} x_{i_j} e_{i_1} \wedge \dots \wedge \widehat{e_{i_j}} \wedge \dots \wedge e_{i_p}
\end{align*}
where $\widehat{\ }$ denotes the deletion of the item. It is easy to verify that $d_{p-1} \circ d_p = 0$. Therefore, there is a finite complex, called the \emph{Koszul complex}:
$$
K_{\bullet}(\textbf{x}): \quad 0 \rightarrow K_n(\textbf{x}) \xrightarrow{d_n} K_{n-1}(\textbf{x}) \rightarrow \cdots \rightarrow K_1(\textbf{x}) \xrightarrow{d_1} K_0(\textbf{x}) \rightarrow 0
$$
of finitely generated free $R$-modules. 

For any $R$-module $M$, define $K^{\bullet}(\textbf{x}, M)$ (resp., $K_{\bullet}(\textbf{x}, M)$) to be the complex $\Hom_R(K_{\bullet}(\textbf{x}), M)$ (resp., $K_{\bullet}(\textbf{x}) \otimes_R M$). The $p$-th homologies, called \emph{Koszul cohomology} (resp., \emph{Koszul homology}), are denoted by $\H^p(\textbf{x}, M)$ (resp., $\H_p(\textbf{x}, M)$), respectively.
It is easy to verify that
\[
\H_0(\textbf{x}, M) \cong M/\textbf{x}M
\]
and
\[
\H_n(\textbf{x}, M) \cong \{ m \in M \mid \ x_i m = 0 \text{ for all } i \}.
\]

Let $I$ be an ideal of $R$ generated by $\textbf{x}$. Then, the \emph{Koszul grade} of $I$ on $M$ is defined as:
$$
\K.\grade_R(I, M) = \inf\{p \in \mathbb{N} \mid \H^p(\textbf{x}, M) \neq 0\}.
$$
Note that the Koszul grade does not depend on the choice of generating sets of $I$ by \cite[Corollary 1.6.22 and Proposition 1.6.10(d)]{BH93}.

\begin{lemma}\cite[Theorem 4.2.8]{S90} $($Koszul duality$)$\label{Koszulduality}
For every module \( M \), there is an isomorphism of complexes
\[
\xymatrix@C=1.2em@R=1.5em{
0 \ar[r] & K_n(\textbf{x}, M) \ar[r] \ar[d]_{\phi_n}
& \cdots \ar[r]
& K_p(\textbf{x}, M) \ar[r]^{d_p^M} \ar[d]^{\phi_p}
& K_{p-1}(\textbf{x}, M) \ar[r] \ar[d]^{\phi_{p-1}}
& \cdots \ar[r]
& K_0(\textbf{x}, M) \ar[r] \ar[d]^{\phi_0}
& 0 \\
0 \ar[r] & K^0(\textbf{x}, M) \ar[r]
& \cdots \ar[r]
& K^{n-p}(\textbf{x}, M) \ar[r]_{d_M^{n-p+1}}
& K^{n-p+1}(\textbf{x}, M) \ar[r]
& \cdots \ar[r]
& K^n(\textbf{x}, M) \ar[r]
& 0
}
\]
In particular, $\H_p(\textbf{x},M) = \H^{n-p}(\textbf{x},M)$ for all $p$.
\end{lemma}

For an ideal $J$ (not necessarily finitely generated) of $R$, the Koszul grade of $J$ on $M$ is defined as:
$$
\K.\grade_R(J, M) = \sup\{\K.\grade_R(I, M) \mid I \ \text{is a finitely generated subideal of } J\}.
$$ 

\begin{theorem}\cite[Theorem 3.4]{z26}
	Let $R$ be a ring. Then
	$$
	\fPD(R) = \sup\{\K.\grade(\m, R) \mid \m \in \Max(R)\}.
	$$
\end{theorem}

\begin{corollary}\label{main}\cite[Theorem 3.1]{zw23} \cite[Theorem 3.4]{z26}
Let $R$ be a ring and $d\geq 0$. The following conditions are equivalent:
\begin{enumerate}
    \item   $\fPD(R)\leq d$;
     \item  any  finitely generated ideal $I$ that satisfies $\Ext_R^i(R/I,R)=0$ for each $i=0,\dots,d$ is $R$;
     \item $\K.\grade(I,R)\leq d$ for proper finitely generated ideal $I$ of  $R.$    
\end{enumerate}
\end{corollary}
\begin{proof}
	The equivalence of (1) and (2) directly comes from \cite[Theorem 3.1]{zw23}, and  the equivalence of (1) and (3) can be easily deduced by  \cite[Theorem 3.4]{z26}. 
\end{proof}

Note that we cannot directly read off how $\Ext_R^i(R/I,R)$ behaves from Corollary \ref{main}(2) when $i$ is large enough. This problem is now addressed by the following main theorem.

\begin{theorem}\label{wnd}
Let $R$ be a ring and $d\geq 0$. Then $\fPD(R)\leq d$ if and only if for any finitely generated ideal $I$ of $R$, if $\Ext_R^i(R/I,R)=0$ for each $i=0,\dots,d$, then $\Ext_R^i(R/I,R)=0$ for all $i\geq 0.$
\end{theorem}
\begin{proof} Suppose $R$ has the small finitistic dimension at most $d$. Let $I$ be a finitely generated ideal of $R$ satisfying $\Ext_R^i(R/I,R)=0$ for each $i=0,\dots,d$. Then $I=R$ by Corollary \ref{main}. So  $\Ext_R^i(R/I,R)=0$ for all $i\geq 0.$

On the other hand, suppose that for any finitely generated ideal $I$ of $R$, if $\Ext_R^i(R/I,R)=0$ for each $i=0,\dots,d$, then $\Ext_R^i(R/I,R)=0$ for all $i\geq 0.$  Let $I$ be an such ideal of $R$. Then   $\Ext_R^i(R/I,R)=0$ for all $i\geq 0.$  It follows by \cite[Theorem 6.1.6]{S90} that $\K.\grade(I,R)=\infty$.
Assume that $I$ is generated by a finite sequence $\textbf{x}:=x_1,\dots,x_m$ of elements in $I$. Then we have the following Koszul complex generated by $\textbf{x}$
$$
K_{\bullet}(\textbf{x}): \quad 0 \rightarrow K_m(\textbf{x}) \xrightarrow{d_m} K_{m-1}(\textbf{x}) \rightarrow \cdots \rightarrow K_1(\textbf{x}) \xrightarrow{d_1} K_0(\textbf{x}) \rightarrow 0
$$
and Koszul cocomplex $K^{\bullet}(\textbf{x}):=\Hom_R(K_{\bullet}(\textbf{x}),R),$
$$
K^{\bullet}(\textbf{x}):\quad 0 \rightarrow K^0(\textbf{x}) \xrightarrow{d^0} K^{1}(\textbf{x}) \rightarrow \cdots \rightarrow K^{m-1}(\textbf{x}) \xrightarrow{d^{m-1}} K^m(\textbf{x}) \rightarrow 0
$$
It follows by Lemma \ref{Koszulduality} that 
 there is an isomorphism of complexes
\[
\xymatrix@C=1.2em@R=1.5em{
	0 \ar[r] & K_m(\textbf{x}) \ar[r] \ar[d]_{\phi_m}
	& \cdots \ar[r]
	& K_p(\textbf{x}) \ar[r]^{d_p} \ar[d]^{\phi_p}
	& K_{p-1}(\textbf{x}) \ar[r] \ar[d]^{\phi_{p-1}}
	& \cdots \ar[r]
	& K_0(\textbf{x}) \ar[r] \ar[d]^{\phi_0}
	& 0 \\
	0 \ar[r] & K^0(\textbf{x}) \ar[r]
	& \cdots \ar[r]
	& K^{m-p}(\textbf{x}) \ar[r]_{d^{m-p+1}}
	& K^{m-p+1}(\textbf{x}) \ar[r]
	& \cdots \ar[r]
	& K^m(\textbf{x}) \ar[r]
	& 0
}
\]
and so is an isomorphism of  homologies:$$\H_p(\textbf{x})\cong \H^{m-p}(\textbf{x})$$ for each $p=0,\dots,m.$
Since $\K.\grade(I,R)=\infty$ as above, we have $\H^{m}(\textbf{x})=0$ in special. Hence $$R/I\cong \H_0(\textbf{x})\cong \H^{m}(\textbf{x})=0.$$ Consequently, $I=R$ which implies that $\fPD(R)\leq d$ by Corollary \ref{main}.
\end{proof}

\begin{corollary}\label{wnd-inf}
	Let $R$ be a ring. Then $\fPD(R)=\infty$ if and only if for any positive integer $i$ large enough, there exists a  finitely generated ideal $I$ of $R$ such that $\Ext_R^i(R/I,R)\not=0$.
\end{corollary}
\begin{proof} Suppose $\fPD(R)=\infty$. On contrary, assume that there a positive integer $d$ such that $\Ext_R^i(R/I,R)=0$ for any finitely ideal $I$ of $R$ and any integer $i>d$. Then $\fPD(R)\leq d$ by Theorem \ref{wnd}, which is a contradiction.
	
On the other hand, suppose that for any positive integer $i$ large enough, there exists a  finitely generated ideal $I$ of $R$ such that $\Ext_R^i(R/I,R)\not=0$. On contrary, assume that $\fPD(R)=d<\infty$. Then, by Theorem \ref{wnd},  any finitely generated ideal $I$ of $R$ such that $\Ext_R^i(R/I,R)=0$ for each $i=0,\dots,d$ satisfies $\Ext_R^i(R/I,R)=0$ for all $i\geq 0,$ which is also a contradiction.
\end{proof}

\section{Applications}

\subsection{fPD(R) VS FP-id(R)}

Recall that an $R$-module $M$ is called \emph{FP-injective} if $\Ext_R^{1}(N, M) = 0$ for all finitely presented modules $N$. For an $R$-module $M$, the \emph{FP-injective dimension} of $M$, denoted by $\FP\id_R M$, is defined to be the least nonnegative integer $n$ such that $\Ext_R^{n+1}(N, M) = 0$ for all finitely presented $R$-modules $N$. If no such $n$ exists, set $\FP\id_R M = \infty$. The number $\FP\id_{R} R$ is called the  \emph{self-FP-injective dimension} of a ring $R$. If we denote by $\id_{R} R$ the self-injective dimension of a ring $R$, then trivially $\FP\id_{R} R\leq \id_{R} R.$

It follows by \cite[Corollary 5.5]{B62} that $\fPD(R)\leq \id_RR$ when $R$ is a Noetherian ring. Now, we can generalized this result to a general ring. 

\begin{theorem}\label{idRfpd}
	Let $R$ be a ring. Then $\fPD(R)\leq \FP\id_RR$.
\end{theorem}
\begin{proof}
	Assume that $\FP\id_RR=d.$ Let	$I$ be a finitely generated ideal  of $R$ such that $\Ext_R^i(R/I,R)=0$ for each $i=0,\dots,d$.
	Since $\FP\id_RR=d,$ we have $\Ext_R^i(R/I,R)=0$ for all $i\geq d+1$ as $R/I$ is finitely presented $R$-module.
	Hence $\fPD(R)\leq d$ by Theorem \ref{wnd}.	
\end{proof}

\begin{corollary}\label{idRfpd}
	Let $R$ be a ring. Then $\fPD(R)\leq \id_RR$.
\end{corollary}
\begin{proof}
	This is because $\FP\id_RR\leq \Id_RR$ for any ring $R$.
\end{proof}

\begin{remark}
If $R$ is a Noetherian ring,  then $\id_RR=\FP\Id_RR=n<\infty$ if and only if $R$ is $n$-Gorenstein. And in this case,  $\fPD(R)=\FPD(R)=n$. Let $R$ be a non-Gorenstein Noetherian ring with Krull dimension equal to $n<\infty$. Then $n=\FPD(R)<(\FP)\Id_RR=\infty.$
\end{remark}
 The following example also shows that $\id_RR$ may be strictly smaller than  $\FPD(R)$.
\begin{example}
	Let $R = \prod\limits_{\aleph_n} k$ be a direct product of $\aleph_n$ ($n < \infty$) copies of a field $k$. Assume that $2^{\aleph_n} = \aleph_m$ with $n + 1 \leq m < \infty$. Then $\FPD(R) = \gld(R) = m + 1$ (see \cite{O73}). However, since the direct product of self-injective rings is also self-injective, we have  $(\FP)\Id_RR = 0$ for any $n$.
\end{example}

\subsection{fPD of weak (n,d)$\mbox{-}$rings} 
For a nonnegative integer \( n \), an \( R \)-module \( E \) is \( n \)-presented if there is an exact sequence \begin{center}
 \( F_n \to F_{n - 1} \to \cdots \to F_0 \to E \to 0 \),
 \end{center}
  in which each \( F_i \) is a finitely generated free \( R \)-module. In particular, ``0-presented'' means finitely generated and ``1-presented'' means finitely presented.

In 1994, Costa \cite{C94} introduced the concept of $(n,d)$-rings in order to parameterize the families of non-Noetherian rings. Subsequently, Zhou \cite{Z04} and Mahdou \cite{M06} independently proposed two weak versions of $(n,d)$-rings defined in terms of flat dimension and $n$-presented cyclic modules, respectively.

\begin{definition}
For nonnegative integers \( n \) and \( d \), a ring \( R \) is said to be
\begin{enumerate}
\item \cite{C94} an \emph{\( (n, d) \)-ring} if \( \text{pd}_R(E) \leq d \) for each \( n \)-presented \( R \)-module \( E \);
\item \cite{Z04} a \emph{weak \( (n, d) \)-ring in sense of Zhou} if \( \text{fd}_R(E) \leq d \) for each \( n \)-presented \( R \)-module \( E \).
\item \cite{M06} a \emph{weak \( (n, d) \)-ring in sense of Mahdou} if \( \text{pd}_R(E) \leq d \) for each \( n \)-presented cyclic \( R \)-module \( E \).
\end{enumerate}
\end{definition}

 


\begin{proposition}
	Every weak $(n,d)$-ring in sense of Zhou, hence every $(n,d)$-ring,  has the  small finitistic dimension at most $d$.
\end{proposition}
\begin{proof}
	Let $R$ be a weak  $(n,d)$-ring in sense of Zhou, and let $M\in\FPR$. Then there exist an integer $n$ and an exact sequence
	$$0\rightarrow P_m\rightarrow P_{m-1}\rightarrow \dots\rightarrow P_1\rightarrow P_0\rightarrow M\rightarrow 0$$
	with each $P_i$  finitely generated projective.
Then $M$ is a $n$-presented  $R$-module, and hence $\fd_RM\leq d$. So there is an exact sequence $$0\rightarrow F_d\rightarrow P_{d-1}\rightarrow \dots\rightarrow P_1\rightarrow P_0\rightarrow M\rightarrow 0$$
with $F_d$ flat, and so is finitely generated projective as it is finitely presented. Consequently, $\fPD(R)\leq d.$
\end{proof}

\begin{proposition}
Every weak $(1,d)$-ring in sense of Mahdou has the  small finitistic dimension at most $d$.
\end{proposition}
\begin{proof} Let $R$ be a weak $(1,d)$-ring in sense of Mahdou, and $I$ be a finitely generated ideal of $R$. Then $R/I$ is a finitely presented cyclic \( R \)-module. And so $\pd_R(R/I)\leq d$. Assume that $\Ext_R^i(R/I,R)=0$ for each $i=0,\dots,d$ . Then, trivially, $\Ext_R^i(R/I,R)=0$ for all $i\geq 0.$ Consequently, $\fPD(R)\leq d.$
\end{proof}

In general, we propose the following open question:
\begin{question}
Does every weak $(n,d)$-ring in sense of Mahdou  have the small finitistic dimension at most $d$?
\end{question}





	

\subsection{New Characterizations of DW-rings} Recall from \cite{fk16} that  a finitely generated ideal $J$ of $R$ is said to be a \emph{$\GV$-ideal} provided that $\Hom_R(R/J,R)=\Ext_R^1(R/J,R)=0$. The set of all $\GV$-ideals of $R$ is denoted by $\GV(R)$.
For an $R$-module $M$, the \emph{$\GV$-torsion submodule} of $M$ is defined to be
\[
\tor(M)
:=\{\,x\in M \mid \exists\,J\in \GV(R)\ \text{such that } Jx=0\,\}.
\]
If $\tor(M)=M$ (resp., $\tor(M)=0$), then $M$ is called a \emph{$\GV$-torsion module} (resp., \emph{$\GV$-torsion-free module}).  An $\GV$-torsion-free $R$-module is said to be a \emph{$w$-module} if $\Ext_R^1(R/J,M)=0$ for any $J\in\GV(R)$.

A ring $R$ is said to be a \emph{\DW-ring} provided that  $\GV(R)=\{R\}$, or equivalently, the only  $\GV$-torsion $R$-module is $0$, or equivalently every $R$-module is a $w$-module. Examples of $\DW$ rings contain rings of  Krull dimension at most $1$, \Prufer\ domains and so on.  Moreover, the authors homologically characterized the class of $\DW$ rings.

\begin{proposition}\label{DW}\cite[Corollary 3.7]{zw23}
Let $R$ be a ring. Then $\fPD(R)\leq 1$ if and only if $R$ is a $\DW$ ring.
\end{proposition}

Recall from \cite{WQ19} that a $w$-module $M$ is said to be a \emph{strong $w$-module} if $\Ext_R^i(T,R)$ for
each integer $i\geq 2$ and for all GV-torsion modules $T$.

\begin{lemma}\label{allsw}
	Let $R$ be a ring. Then every $R$-module is a strong $w$-module if and only if $R$ is a $\DW$-ring.
\end{lemma}
\begin{proof} Suppose  every $R$-module is a strong $w$-module, and hence a $w$-module. So $R$ is a $\DW$-ring.	On the other hand, assume that $R$ is a $\DW$-ring. Then the only $\GV$-torsion $R$-module is $0$. Hence every $R$-module is a  strong $w$-module.
\end{proof}

Furthermore, we have the following characterization of $\DW$-rings.
\begin{proposition}\label{Rsw}
	Let $R$ be a ring. Then $R$ itself is a  strong $w$-module if and only if $R$ is a \DW-ring.
\end{proposition}
\begin{proof} Suppose $R$ is a \DW-ring. Then it follows by Lemma \ref{allsw} that $R$ itself is a  strong $w$-module.

On the other hand, let a finitely generated ideal $J$ of $R$ be a $\GV$-ideal, equivalently, $\Ext_R^i(R/J,R)=0$ for each $i=0,1$. Then, by assumption, $R$ is a  strong $w$-module.  So $\Ext_R^i(R/J,R)$ for	each integer $i\geq 0$. It follows by Theorem \ref{wnd} that $\fPD(R)\leq 1$. So $R$ is a  DW-ring by Proposition \ref{DW}.
\end{proof}

\begin{corollary}
	Let $R$ be a ring. Then  every $w$-module is a  strong $w$-module if and only if $R$ is a \DW-ring.
\end{corollary}
\begin{proof}
	It follows by Proposition \ref{Rsw} and $R$ always is a $w$-module.	
\end{proof}

Recently, Zhou \cite[Proposition 2]{ZKH26} showed that an integer domain $R$ with $\id_RR\leq 1$ is a $\DW$-ring. Now, we can generalize it to general commutative rings.
\begin{corollary}\label{idDW}
Let $R$ be a ring. If $\FP\id_RR\leq 1$, then $R$ is a $\DW$-ring.
\end{corollary}
\begin{proof}
	It follows by Theorem \ref{idRfpd} and Proposition \ref{DW}.
\end{proof}

\begin{remark}
	The converse of Corollary \ref{idDW} is not true in general. Indeed, let $R=k[[x,y]]/(x^2,xy,y^2)$ with $k$ a field. Then $R$ is not Gorenstein, and so $\FP\id_RR=\id_RR=\infty$. However, K.$\dim(R)$=0,  so $\fPD(R)=0$ by \cite[Theorem 3.5]{z26}, and hence $R$ is a $\DW$-ring. This example also shows that $\fPD(R)$ and $\FP\id_RR$ can differ largely (see Theorem \ref{idRfpd}).
\end{remark}
\begin{remark}
	The main results of this subsection can be formally extend to $\DQ$-rings and strong Lucas modules (see \cite{z26-tqps}). Since the results and proofs are similar to those presented in this subsection, we omit them here.
\end{remark}

\subsection{Srong \Prufer\  rings VS \Prufer\ rings}  Recall  that a ring $R$ is said to be a \emph{\Prufer\ ring} (resp., \emph{strong \Prufer\ ring}) provided that every finitely generated regular (resp., semi-regular) ideal is projective. It was showed in \cite[Example 3.10]{zw23} and \cite[Example 6.4]{z26} that, for any $n\in\mathbb{N}\cup \{\infty\},$ there is a \Prufer\ ring has the small finitistic dimension $n$. However, it was proved in \cite{zw23} that every strong \Prufer\ ring has the small finitistic dimension at most $1$.  So we have the following result.

\begin{corollary}
Every strong \Prufer\ ring  is a  \Prufer\ $\DW$-ring.
\end{corollary}

A natural question is: 
\begin{center}
\textbf{Question:} 	Is every \Prufer\ $\DW$-ring a strong \Prufer\ ring?
\end{center}
Finally, we deny this question by a counterexample.

\begin{example}
Let $D=k[x,y]$ be the polynomial ring in two variables over a field $k$, and let $Q$ be its quotient field. Set $R=D(+)Q/D$ the idealization of $D$ by $Q/D$ (see \cite{AW09}).  Then $R$ is a total ring of quotients, and hence a \Prufer\ ring. It follows by by \cite[Corollary 6.3]{z26} that 
$\fPD(R)=1$ , and so $R$ is a $\DW$-ring by Proposition \ref{DW}.

Next, we will show that $R$ is not a strong \Prufer\ ring.  We claim that the maximal ideal $\m:=\langle x,y\rangle (+)Q/D$ of $R$ is non-projective semiregular. Indeed, assume $\m(r,\frac{a}{b}+D)=0$, then $r=0$, and $b|x$ and $b|y$. So $b$ is a unit in $D$, and hence $(r,\frac{a}{b}+D)=0$. Consequently, $\m$ is semiregular. Since $\m_{\m}$ is not a principal ideal of $R_{\m}$. Then the finitely generated semiregular ideal $\m$ of $R$ is not projective (or equivalently $Q_0$-invertible) (see \cite[Page 64]{L93}). So $R$ is not a strong \Prufer\ ring. 

In conclusion, $R$ is a \Prufer\ $\DW$-ring but not a strong \Prufer\ ring.
\end{example}

\end{document}